\documentclass[a4paper,12pt,reqno]{amsart}
\usepackage{amsfonts}
\usepackage{amsmath}
\usepackage{amssymb}
\usepackage[a4paper]{geometry}
\usepackage{mathrsfs}
\usepackage{csquotes}
\usepackage{enumitem}
\usepackage{dirtytalk}
\usepackage{comment}
\usepackage{nccmath}

\usepackage[colorlinks]{hyperref}
\renewcommand\eqref[1]{(\ref{#1})} %Need with hyperref
\numberwithin{equation}{section}
\theoremstyle{plain}
\newtheorem{thm}{Theorem}[section]

\newtheorem{cor}[thm]{Corollary}

\theoremstyle{definition}

\newtheorem{rem}[thm]{Remark}

\newcommand\norm[1]{\left\lVert#1\right\rVert}

\def\e[#1]{{\textrm{e}}^{#1}}

\begin{document}

   \title[Sharp remainder of the $L^{p}$-Poincar\'e inequality]
 {Sharp remainder of the $L^{p}$-Poincar\'e inequality for Baouendi-Grushin vector fields}

\author[K. Apseit]{Kuralay Apseit}
\address{
  Kuralay Apseit:
  \endgraf
  SDU University, Kaskelen, Kazakhstan
     \endgraf
  and 
  \endgraf
  Institute of Mathematics and Mathematical Modeling, Kazakhstan
  \endgraf
  {\it E-mail address} {\rm kuralayapseit@gmail.com}
  }

\author[N. Yessirkegenov]{Nurgissa Yessirkegenov}
\address{
  Nurgissa Yessirkegenov:
  \endgraf
  KIMEP University, Almaty, Kazakhstan
     \endgraf
  {\it E-mail address} {\rm nurgissa.yessirkegenov@gmail.com}
  }

  \author[A. Zhangirbayev]{Amir Zhangirbayev}
\address{
  Amir Zhangirbayev:
 \endgraf
   SDU University, Kaskelen, Kazakhstan
  \endgraf
  and 
  \endgraf
  Institute of Mathematics and Mathematical Modeling, Kazakhstan
   \endgraf
  {\it E-mail address} {\rm
amir.zhangirbayev@gmail.com}
  }

\thanks{This research is funded by the Committee of Science of the Ministry of Science and Higher Education of the Republic of Kazakhstan (Grant No. AP23490970).}


     \begin{abstract} In this paper, we establish a sharp remainder formula for the Poincar\'e inequality for Baouendi-Grushin vector fields in the setting of $L^{p}$ for complex-valued functions. In special cases, we recover previously known results. Consequently, we also derive the $L^{p}$-Poincar\'e inequality with an explicit optimal constant under a certain assumption. Additionally, we provide estimates of the remainder term for $p\geq2$ and $1<p<2\leq n<\infty$. As an application, we obtain a blow-up in finite time and global existence of the positive solutions to the initial-boundary value problem of the doubly nonlinear porous medium equation involving a degenerate nonlinear operator $\Delta_{\gamma,p}$.  
     
     \end{abstract}
     \maketitle

\section{Introduction}

The Poincar\'e inequality plays a crucial role in many questions from nonlinear partial differential equations, spectral theory and geometric analysis. The classical Poincar\'e inequality \cite[Chapter 5.8.1]{evans2022partial} reads as follows: let $1\leq p<\infty$ and $\Omega$ be any bounded subset of the domain $\mathbb{R}^{n}$. Then, there exists a constant $C$ depending only on $\Omega$ and $p$ such that for every $u$ in the Sobolev space $W_{0}^{1,p}(\Omega)$ of zero-trace functions, we have
\begin{align}\label{Poincare}
\norm{u}_{L^{p}(\Omega)}\leq C\norm{\nabla u}_{L^{p}(\Omega)}.
\end{align}
The inequality (\ref{Poincare}) has been extensively studied, and we refer the reader to \cite{payne1960optimal, jerison1986poincare, lu1992weighted, hurri1994improved, lu1994sharp, franchi1995representation, franchi1996relationship, semmes1996finding, bebendorf2003note, lieb2003poincare, lott2007weak, boulkhemair2007uniform, keith2008poincare, li2015poincare, ozawa2020sharp, bobkov2023improved, suraganozawafred, do2024scale} for notable results and developments.

There has also been a significant interest in Poincar\'e inequalities on the sub-Riemannian space $\mathbb{R}^n=\mathbb{R}^{m}\times\mathbb{R}^{k}$ defined by the Baouendi-Grushin vector fields:
\begin{align*}
X_i = \frac{\partial}{\partial x_i}, \quad i = 1, \ldots, m, \qquad
Y_j = |x|^{\gamma} \frac{\partial}{\partial y_j}, \quad j = 1, \ldots, k, 
\end{align*}
where $x=(x_1, \ldots, x_m)\in\mathbb{R}^m$, $y=(y_1, \ldots, y_k)\in\mathbb{R}^{k}$ with $m,k\geq1$ and $\gamma\geq0$. For example, in \cite{franchi1994two, franchi1994weighted} Franchi, Guti\'errez and Wheeden extended the class of weight functions for which the Sobolev-Poincar\'e inequalities are known to hold, thereby obtaining the following inequality for metric balls related to the generalized Grushin differential operator $\Delta_{\mu}=\Delta_{x}+\mu^2(x)\Delta_{y}$:
\begin{align}\label{fgw}
\left( \frac{1}{w_2(B)} \int_B |u|^q w_2(z) \, dz \right)^{1/q}\leq cr\left( \frac{1}{w_1(B)} \int_B |\nabla_{\mu} u|^p w_1(z) \, dz \right)^{1/p},
\end{align}
where $u\in C^{\infty}_{0}(B)$ and the weight functions $w_{1}$, $w_{2}$ satisfy a specific condition. The constant $c$, in (\ref{fgw}), is independent of $u$ and $B$, $1\leq p\leq q < \infty$ and $w(B)=\int_{B}w(z)dz$. We note that the inequality (\ref{fgw}), for $w_1(B)=w_{2}(B)=1$, implies the Sobolev inequality, and we refer to the results of Monti \cite{monti2006sobolev} for related findings.

Then, D'Ambrosio \cite[Theorem 3.7]{d2004hardy}, mainly using techniques developed in \cite{mitidieri2000simple, d2004hardy}, proved Poincar\'e inequality on domains $\Omega$ contained in a slab: let $\Omega$ be an open set on $\mathbb{R}^n$. Suppose that there exists $R>0$, a real number $s$ and an integer $1\leq j \leq m$ such that for any $z=(x,y)\in\Omega$, it follows that $|x_j-s|\leq R$. Then, for all $u\in C^{1}_{0}(\Omega)$, we have
\begin{align*}
\int_{\Omega}|u|^{p}dz\leq c\int_{\Omega}|\nabla_{\gamma}u|^{p}dz,   
\end{align*}
where $c=(pR)^{p}$.

Another interesting result in this direction, to which we will return later, was obtained by Suragan and the second author in \cite[Corollary 1.2]{suragan2023sharp} by employing a simple method from \cite{ozawa2020sharp} that does not involve the variational principle: suppose that the minus Dirichlet Baouendi-Grushin operator on $\Omega$ has a positive eigenvalue $\lambda$ and a corresponding positive eigenfunction $\phi$. Then, we have
\begin{align*}
\frac{1}{\lambda} \int_{\Omega} \left| \nabla_{\gamma} u - \frac{\nabla_{\gamma} \phi}{\phi} u \right|^2dz
= \frac{1}{\lambda} \int_{\Omega} |\nabla_{\gamma} u|^2dz - \int_{\Omega} |u|^2dz
\end{align*}
for all $u\in W^{1,2}_{\gamma}(\Omega)$. 

Most recently, D'Arca \cite[Theorem 4.4]{d2024weighted} derived the following Poincar\'e weigh-ted inequalities while avoiding the symmetric rearrangement argument \cite{d2024unified}, thereby simplifying the analysis in Euclidean and non-Euclidean contexts: let \( p \geq 2 \), \( \alpha \geq 0 \), and \( \theta \geq 1 \) be fixed. For all \( u \in W_{\gamma}^{1,p}(B_R^{\rho}, |\nabla_\gamma \rho|^{\alpha} \rho^{\theta - Q}) \), the following inequalities hold:
\begin{multline*}
\left( \frac{\nu_1(p, \theta)}{R} \right)^p \int_{B_R^\rho} \frac{|u|^p}{\rho^{Q - \theta}} |\nabla_\gamma \rho|^{\alpha + p} \, dz \leq 
\int_{B_R^\rho} \left| \nabla_\gamma \rho \cdot \frac{\nabla_\gamma u}{|\nabla_\gamma \rho|} \right|^p \frac{|\nabla_\gamma \rho|^{\alpha}}{\rho^{Q - \theta}} \, dz
\\\leq 
\int_{B_R^\rho} \frac{|\nabla_\gamma u|^p}{\rho^{Q - \theta}} |\nabla_\gamma \rho|^{\alpha} \, dz.
\end{multline*}
Moreover, the chain of inequalities is sharp since the function \( u = \varphi\left( \frac{\nu_1(p, \theta)}{R} \rho \right) \in W_{\gamma}^{1,p}(B_R^{\rho}, |\nabla_\gamma \rho|^{\alpha} \rho^{\theta - Q}) \) attains both equalities. Here, $\nu_1(p,\theta)$ is the first zero of an appropriate special function and $B^{\rho}_{R}=\{z\in\mathbb{R}^n:\rho(z)<R\}$. 

The purpose of this paper, however, is to extend the results of Suragan and the second author \cite{suragan2023sharp} from $p=2$ to any $1<p<\infty$. More precisely, we obtain the following identity: let $1<p<\infty$ and $\Omega\subset\mathbb{R}^{m+k}$ be a set supporting the divergence formula. Then, for all complex-valued $u\in W^{1,p}_{\gamma}(\Omega)$ and all non-zero, complex-valued and twice differentiable $\phi$, we have
\begin{align}\label{main identity intro}
\int_{\Omega}^{}C_p\left(\nabla_\gamma u,\nabla_\gamma u - \frac{\nabla_\gamma \phi}{\phi} u\right)dz = \int_{\Omega}^{}|\nabla_{\gamma}u|^{p}dz+\int_{\Omega}^{}\frac{|u|^{p}}{|\phi|^{p-2}\phi}\Delta_{\gamma,p}\phi dz,
\end{align}
where the functional $C_p(\cdot,\cdot)$ is given by
\begin{align}\label{cp functional intro}
C_p(\xi, \eta) = |\xi|^p - |\xi - \eta|^p - p |\xi - \eta|^{p-2} \textnormal{Re}  (\xi - \eta) \cdot \overline{\eta}\geq0.
\end{align}

Assuming that the negative Dirichlet $p$-Laplace Grushin ($p$-Grushin) operator on some bounded open subset $D\subset \mathbb{R}^{m+k}$ has a positive eigenvalue $\lambda$ with an associated positive eigenfunction $\phi$, we get the following sharp remainder formula of the $L^{p}$-Poincar\'e inequality for Baouendi-Grushin vector fields from (\ref{main identity intro}):
\begin{align}\label{main poincare intro}
\int_{D}C_p\left(\nabla_\gamma u,\nabla_\gamma u - \frac{\nabla_\gamma \phi}{\phi} u\right)dz=\int_{D}|\nabla_{\gamma}u|^{p}dz-\lambda\int_{D}|u|^{p}dz
\end{align}
for all complex-valued $u\in W^{1,p}_{\gamma}(D)$. Immediately, we see that the identity (\ref{main identity intro}) together with (\ref{main poincare intro}) generalizes \cite[Theorem 1.1]{suragan2023sharp} and \cite[Corollary 1.2]{suragan2023sharp} for any $1<p<\infty$, respectively. In addition, due to the presence of the $C_p$-functional in both identities, (\ref{main identity intro}) and (\ref{main poincare intro}), we are able to provide a simple characterization of nontrivial extremizers and their existence by \cite[Step 3 of Proof of Lemma 3.4]{cazacu2024hardy} and \cite[Lemma 2.2 and 2.3]{CT24}. Moreover, the same results with \cite[Lemma 2.4]{CT24} allow us to obtain estimates of the remainder term for $p\geq2$ (see Corollary \ref{cor laptev}) and $1<p<2\leq n<\infty$ (see Corollary \ref{cor chinese}). As a result, we recover the Poincar\'e improvement of Bobkov and Kolonitskii \cite[Theorem 1.5, Equation (2.10)]{bobkov2023improved} with an explicit constant.

If we set $\lambda=\lambda_1>0$ to be the first eigenvalue of $-\Delta_{\gamma,p}$ with an associated (presumed positive) eigenfunction $\phi=\phi_1$ on $D$, then, from (\ref{main poincare intro}), we are able to obtain the following $L^{p}$-Poincar\'e inequality for all complex-valued $u\in W^{1,p}_{\gamma}(D)$:
\begin{align}\label{main poincare l1 intro}
\int_{D}|u|^{p}dz\leq\frac{1}{\lambda_1}\int_{D}|\nabla_{\gamma}u|^{p}dz,
\end{align}
where the constant $\frac{1}{\lambda_1}$ is optimal and attained if and only if $\frac{u}{\phi_1}=\text{const}$.

In this paper, we also discuss applications of (\ref{main poincare l1 intro}) to the study of blow-up and global existence of the positive solutions to the initial boundary value problem of the doubly nonlinear porous medium equation (PME) related to the $p$-Grushin operator $\Delta_{\gamma,p}$:
\begin{align}\label{pde}
\begin{cases}
u_t - \Delta_{\gamma,p}(u^{\ell}) = f(u), & z \in D,\ t > 0, \\
u(z, t) = 0, & z \in \partial D,\ t > 0, \\
u(z, 0) = u_0(z) \geq 0, & z \in \overline{D}.
\end{cases}
\end{align}
Here, $D$ is an open bounded domain of $\mathbb{R}^{m+k}$, $\ell\geq1$, $f$ is locally Lipschitz continuous on $\mathbb{R}$, satisfies $f(0)=0$ and is strictly positive for all $u>0$. The initial condition $u_0$ is a positive function in $C^{1}(\overline{D})$ such that $u_0(z)=0$ for all $z\in\partial D$.

The PME is an important example of a nonlinear evolution equation of parabolic type. It arises in the modeling of various natural processes, such as fluid motion, heat transfer and diffusion. One of the most well-known examples is the modeling of isentropic gas flow through porous medium, developed independently by Leibenzon \cite{leibenzon1930motion} and Muskat \cite{muskat1938flow}. Another important application is in the study of radiation occurring in plasmas (ionized gases) at very high temperatures, pioneered by Zel’dovich and Raizer \cite{zel2002physics}. In fact, this application contributed significantly to the mathematical theory’s development. We refer to the V\'azquez’s book \cite{vazquez2007porous} for an extensive coverage of the theory of PMEs.

Recently, Poincar\'e inequality was used to study blow-up and global existence properties of the positive solutions of PMEs and other kinds of partial differential equations \cite{suragan2023sharp, ruzhansky2023global,  dukenbayeva2024global, sabitbek2024global, dukenbayeva2025global, jabbarkhanov2025blow, jabbarkhanov2025global}. In this paper, we extend the results of Dukenbayeva \cite[Theorem 1.3 and 1.6]{dukenbayeva2024global} from $p=2$ to any $1<p<\infty$. 

The paper is organized as follows. Section \ref{prem} introduces basic definitions, notation and preliminary results. In Section \ref{main results}, we prove the sharp remainder formula of $L^{p}$-Poincar\'e inequality for Baouendi-Grushin vector fields. As a result, we also derive the $L^{p}$-Poincar\'e inequality with an explicit optimal constant under a particular condition. Additionally, we show that the obtained results imply the estimates of the remainder term. Finally, in Section \ref{applications}, applications to the initial boundary value problem of the doubly nonlinear PME are considered.

\section{Preliminaries}\label{prem}

In this section, we recall some notation and preliminary results regarding the Baouendi-Grushin operator and Sobolev spaces. 

Let $z=(x_{1},\ldots,x_{m},y_{1},\ldots,y_{k})$ or $z=(x,y)$ $\in\mathbb{R}^{m}\times\mathbb{R}^{k}$ with $m+k=n$ and $m,k\geq1$. The sub-elliptic gradient is defined as
\begin{align*}
\nabla_\gamma = (X_1, \ldots, X_m, Y_1, \ldots, Y_k) = (\nabla_x, |x|^{\gamma} \nabla_y),
\end{align*}
where
\begin{align}\label{grushin vector fields}
X_i = \frac{\partial}{\partial x_i}, \quad i = 1, \ldots, m, \qquad
Y_j = |x|^{\gamma} \frac{\partial}{\partial y_j}, \quad j = 1,\ldots,k 
\end{align}
with $\gamma \geq 0$ and $|x| = \left( \sum_{i=1}^{m} x_i^2 \right)^{1/2}$ represents the standard Euclidean norm of $x$. The Baouendi-Grushin operator $\Delta_\gamma$ is a differential operator on $\mathbb{R}^{m+k}$ defined by
\begin{align}\label{grushin}
\Delta_\gamma := \sum_{i=1}^m X_i^2 + \sum_{j=1}^k Y_j^2 = \Delta_x + |x|^{2\gamma} \Delta_y = \nabla_\gamma \cdot \nabla_\gamma.
\end{align}
Here, when $\gamma = 0$, in (\ref{grushin}), the Baouendi-Grushin operator $\Delta_\gamma$ reduces to the classical Laplacian on $\mathbb{R}^{m+k}$. The $p$-Grushin of a complex-valued function $\phi$ on $\Omega\subset\mathbb{R}^{m+k}$ associated with the vector fields (\ref{grushin vector fields}) is defined by
\begin{align*}
\Delta_{\gamma,p}\phi=\nabla_{\gamma}\cdot(|\nabla_{\gamma}\phi|^{p-2}\nabla_{\gamma}\phi), \quad 1<p<\infty.
\end{align*}
When $\gamma$ is an even positive integer, $\Delta_\gamma$ can be expressed as a sum of squares of smooth vector fields satisfying Hörmander's condition on the Lie algebra \begin{align}
\operatorname{rank} \operatorname{Lie}\left[X_1, \ldots, X_m, Y_1, \ldots, Y_k\right]=n \text {. } \nonumber
\end{align}
There is a natural family of anisotropic dilations associated with $\Delta_\gamma$:
\begin{align*}
\delta_{a}(x,y) := (a x, a^{\gamma+1} y), \quad a > 0, \ (x,y) \in \mathbb{R}^{m+k}.
\end{align*}
This shows that the degeneracy of $\Delta_\gamma$ becomes more severe as $\gamma \to \infty$. The corresponding change of variable formula for the Lebesgue measure is:
\begin{align*}
d \circ \, \delta_{a}(x,y) = a^Q \, dx \, dy,
\end{align*}
where the dilation's homogeneous dimension is given by
\begin{align*}
Q = m + k(\gamma + 1).
\end{align*}

Let $\Omega \subset \mathbb{R}^{m+k}$ be a set that supports the divergence formula. The Sobolev space $W^{1,p}_{\gamma}(\Omega)$ associated with Baouendi-Grushin vector fields (\ref{grushin vector fields}) is defined by
\begin{align*}
W^{1,p}_{\gamma}(\Omega):=\{u\in L^{p}(\Omega): \nabla_{\gamma}u\in L^{p}(\Omega)\}.
\end{align*}
We note that $W^{1,p}_{\gamma}(\Omega)$ is the closure of $C^{\infty}_{0}(\Omega)$ in the norm 
\begin{align*}
\norm{u}_{W^{1,p}_{\gamma}(\Omega)} = \left( \int_{\Omega} |\nabla_{\gamma} u|^pdz \right)^{\frac{1}{p}}.    
\end{align*}

\section{Main results}\label{main results}

In this section, we prove the sharp remainder formula of the $L^{p}$-Poincar\'e inequality for Baouendi-Grushin vector fields, show the derivation of the inequality with an explicit optimal constant, and provide the estimates of the remainder term covering the full range of $1<p<\infty$.

\begin{thm}\label{main thm}
Let $1<p<\infty$ and $\Omega\subset\mathbb{R}^{m+k}$ be a set supporting the divergence formula. 
\begin{enumerate}
    \item Then, for all complex-valued $u\in W^{1,p}_{\gamma}(\Omega)$ and all non-zero, complex-valued and twice differentiable $\phi$, we have
\begin{align}\label{main identity}
\int_{\Omega}^{}C_p\left(\nabla_\gamma u,\nabla_\gamma u - \frac{\nabla_\gamma \phi}{\phi} u\right)dz = \int_{\Omega}^{}|\nabla_{\gamma}u|^{p}dz+\int_{\Omega}^{}\frac{|u|^{p}}{|\phi|^{p-2}\phi}\Delta_{\gamma,p}\phi dz,
\end{align}
where the functional $C_p(\cdot,\cdot)$ is given in (\ref{cp functional intro}).
\item Furthermore, for $1<p<\infty$, the $C_p$-functional vanishes if and only if $\frac{u}{\phi}=\text{const}$.
\end{enumerate}
\end{thm}

\begin{rem}
If $p=2$, in (\ref{main identity}), then we recover the result of Suragan and the second author \cite[Theorem 1.1]{suragan2023sharp}:
\begin{align*}
\int_{\Omega} \left| \nabla_{\gamma} u - \frac{\nabla_{\gamma} \phi}{\phi} u \right|^2dz 
= \int_{\Omega}|\nabla_{\gamma} u|^2dz + 
\int_{\Omega}\frac{|u|^2}{\phi}\Delta_{\gamma} \phi dz
\end{align*}
for all $u\in W_{\gamma}^{1,2}(\Omega)$.
\end{rem}

\begin{rem}
Let $D$ be a bounded open subset of $\mathbb{R}^{m+k}$. The spectral problem we consider is of the form
\begin{align}\label{eigen}
\begin{cases}
- \Delta_{\gamma,p} \phi = \lambda |\phi|^{p-2} \phi & \text{in } D, \\
\phi = 0 & \text{on } \partial D,
\end{cases}
\end{align}
where $\lambda\in\mathbb{R}$ is the eigenvalue of the problem if (\ref{eigen}) admits a nontrivial weak solution $\phi \in W^{1,p}_{\gamma}(D)\backslash \{\phi=0\}$. In case when $\gamma=0$ in (\ref{eigen}), it is known that the first eigenvalue is positive and has an associated positive eigenfunction (see, e.g.
\cite{lindqvist1990equation, lindqvist1992addendum}). When $p=2$, (\ref{eigen}) reduces to the classical eigenvalue problem associated with the Baouendi-Grushin operator:
\begin{align*}%\label{eigen p=2}
\begin{cases}
- \Delta_{\gamma} \phi = \lambda \phi & \text{in } D, \\
\phi = 0 & \text{on } \partial D
\end{cases}
\end{align*}
with $\phi\in W^{1,2}_{\gamma}(D)\backslash\{\phi=0\}$. In \cite[Theorem 1]{xu2023nontrivial}, Xu, Chen and O'Regan showed that the spectrum of the $-\Delta_{\alpha}$-Laplacian (that recovers the Baouendi-Grushin operator $\Delta_{\gamma}$) consists of a discrete set of positive eigenvalues $\{\lambda_s\}_{s\in\mathbb{N}}$ of finite multiplicity with
\begin{align*}
0<\lambda_1<\lambda_2\leq\ldots\leq\lambda_{s}\leq\lambda_{s+1}\leq\ldots\rightarrow + \infty, \quad \text{as } s\rightarrow +\infty.
\end{align*}
Moreover, there exists a positive function $\phi_1\in W^{1,2}_{\alpha}(D)$, which is an eigenfunction corresponding to the positive eigenvalue $\lambda_{1}$. We also refer to \cite[Theorem 6.4]{monticelli2009maximum} for related results. Despite this, there is a little information regarding the spectral properties of the $p$-Grushin operator. In a very recent work, Malanchini, Bisci and Secchi \cite[Proposition 4.2]{malanchini2025bifurcation}, defined a non-decreasing sequence $\{\lambda_r\}_{r\in\mathbb{N}}$ of eigenvalues of $-\Delta_{\gamma,p}$ by using the $\mathbb{Z}_2$-cohomological index of Fadell
and Rabinowitz \cite{fadell1977generalized}. Consequently, the authors showed that the sequence $\{\lambda_r\}_{r\in\mathbb{N}}$ diverges to infinity as $r\rightarrow +\infty$ and that the first eigenvalue $\lambda_1$ of $-\Delta_{\gamma,p}$ is the smallest strictly positive eigenvalue. However, the sign of the corresponding eigenfunction $\phi_1$ appears to be unknown. Thus, we will assume that the eigenfunction $\phi_1$ (corresponding to the positive eigenvalue $\lambda_1$ of $-\Delta_{\gamma,p}$) is, in fact, strictly positive as well. 
\end{rem}

\begin{cor}\label{main cor}
Suppose that the minus Dirichlet $p$-Grushin operator $-\Delta_{\gamma,p}$ on $D$ has a positive eigenvalue $\lambda$ and a corresponding positive eigenfunction $\phi$. Then,
\begin{enumerate}
    \item for all complex-valued $u\in W^{1,p}_{\gamma}(D)$, we have
\begin{align}\label{main cor eq}
\int_{D}C_p\left(\nabla_\gamma u,\nabla_\gamma u - \frac{\nabla_\gamma \phi}{\phi} u\right)dz=\int_{D}|\nabla_{\gamma}u|^{p}dz-\lambda\int_{D}|u|^{p}dz
\end{align}
with functional $C_p(\cdot,\cdot)$ is given in (\ref{cp functional intro}).
\item Let $\lambda=\lambda_1>0$ be the first eigenvalue of $-\Delta_{\gamma,p}$ with an associated (presumed positive) eigenfunction $\phi=\phi_1$ on $D$. Then, for all complex-valued $u\in W^{1,p}_{\gamma}(D)$, we have 
\begin{align}\label{poincare lp ineq}
\int_{D}|u|^{p}dz\leq\frac{1}{\lambda_1}\int_{D}|\nabla_{\gamma}u|^{p}dz,
\end{align}
where the constant $\frac{1}{\lambda_1}$ is optimal and attained if and only if $\frac{u}{\phi_1}=\text{const}$.
\end{enumerate}
\end{cor} 

Since for $\gamma=0$, it is proven that the first eigenvalue of the minus $p$-Laplacian operator $-\Delta_{p}$ is positive and its associated eigenfunction is also positive \cite{lindqvist1990equation, lindqvist1992addendum}, we have the following sharp remainder formula of the $L^{p}$-Poincar\'e inequality including the inequality with an optimal constant:

\begin{cor}
Suppose that the minus Dirichlet $p$-Laplacian operator $-\Delta_{p}$ on $D$ has a positive eigenvalue $\lambda$ and a corresponding positive eigenfunction $\phi$. Then,
\begin{enumerate}
    \item for all complex-valued $u\in W^{1,p}_{0}(D)$, we have
\begin{align*}
\int_{D}C_p\left(\nabla u,\nabla u - \frac{\nabla \phi}{\phi} u\right)dz=\int_{D}|\nabla u|^{p}dz-\lambda\int_{D}|u|^{p}dz
\end{align*}
with functional $C_p(\cdot,\cdot)$ is given in (\ref{cp functional intro}).
\item Let $\lambda=\lambda_1>0$ be the first eigenvalue of $-\Delta_{p}$ with an associated positive eigenfunction $\phi=\phi_1>0$ on $D$. Then, for all complex-valued $u\in W^{1,p}_{0}(D)$, we have 
\begin{align*}
\int_{D}|u|^{p}dz\leq\frac{1}{\lambda_1}\int_{D}|\nabla u|^{p}dz,
\end{align*}
where the constant $\frac{1}{\lambda_1}$ is optimal and attained if and only if $\frac{u}{\phi_1}=\text{const}$.
\end{enumerate}
\end{cor} 

\begin{rem}
In the special case, when $p=2$, in (\ref{main cor eq}),  we recover the result of Suragan and the second author \cite[Corollary 1.2]{suragan2023sharp}:
\begin{align*}
\frac{1}{\lambda} \int_{D} \left| \nabla_{\gamma} u - \frac{\nabla_{\gamma} \phi}{\phi} u \right|^2dz
= \frac{1}{\lambda} \int_{D} |\nabla_{\gamma} u|^2dz - \int_{D} |u|^2dz
\end{align*}
for all $u\in W^{1,2}_{\gamma}(D)$.
\end{rem}

\begin{rem}
When $p=2$ and $\gamma=0$, the identities (\ref{main identity}) and (\ref{main cor eq}) reduce to the results of Ozawa and Suragan \cite[Theorem 2.1]{ozawa2020sharp}.
\end{rem}

\begin{rem}
We note that a version of the sharp remainder formula of the $L^{2^m}$-Poincar\'e inequality was established in \cite[Theorem 3.3]{suraganozawafred}: let $\Omega\subset\mathbb{R}^n$ be a connected domain, for which the divergence theorem is true, then we have
\begin{align}\label{fredsuragan}
&\int_{\Omega}|\nabla u|^{p_m} dz-\left(\lambda_1-\sigma_m\right) \int_{\Omega}|u|^{p_m} dz=  \sum_{j=1}^{m-1} \int_{\Omega}\left| \left| 
\nabla(u^{p_{m-j-1}}) \right|^{p_{j}}-2^{p_{j}-1}u^{p_{m-1}}  \right|^{2}dx \nonumber
\\&+\int_{\Omega}^{}\left| \nabla(u^{p_{m-1}})-\frac{\nabla \phi_1}{\phi_1}u^{p_{m-1}} \right| ^{2}dz 
\end{align}
for all $u \in C_0^1(\Omega)$. Here, $\sigma_m = \frac{1}{4} \sum_{j=1}^{m-1} 4^{p_j}$, $m \in \mathbb{N}$, $p_j = 2^j$ and $\phi_1$ is the ground state of the minus Laplacian in $\Omega$ and $\lambda_{1}$ is the corresponding eigenvalue.

By taking $p=2$ and $\gamma=0$ in (\ref{main cor eq}), we recover the $L^{2}$ case of (\ref{fredsuragan}). However, for $p=2^{m}$ (with $m=2,3,\ldots$) and $\gamma=0$, the results (\ref{main cor eq}) and (\ref{fredsuragan}) do not coincide. This is due to the fact that, for $p=2^{m}$ (with $m=2,3,\ldots$) and $\gamma=0$ in (\ref{main cor eq}), $\lambda$ and $\phi$ correspond to the eigenvalues and eigenfunctions of the nonlinear minus $2^{m}$-Laplacian, whereas, in (\ref{fredsuragan}), $\lambda_1$ and $\phi_1$ are always the eigenvalue and eigenfunction of the standard minus Laplacian.  
\end{rem}

Applying the results regarding the estimate of the remainder term, $C_p$-functional, for $p\geq2$ and $1<p<\infty$ from \cite[Step 3 of Proof of Lemma 3.4]{cazacu2024hardy} and \cite[Lemma 2.2, 2.3 and 2.4]{CT24}, respectively, we have the following corollaries:
\begin{cor}\label{cor laptev}
Let $p\geq2$ and $\lambda, \phi, D$ be from Corollary \ref{main cor}. Then, for all complex-valued $u\in W^{1,p}_{\gamma}(D)$, we have
\begin{align}\label{bobkov}
\int_{D}^{}|\nabla_{\gamma}u|^{p}dz-\lambda\int_{D}^{}|u|^{p}dz \geq c_p\int_{D}\left|\nabla_\gamma u - \frac{\nabla_\gamma \phi}{\phi} u\right|^{p}dz,
\end{align}
where
\begin{align*}%\label{laptev cp const}
c_p
= \inf_{(s,t)\in\mathbb{R}^2\setminus\{(0,0)\}}
\frac{\bigl[t^2 + s^2 + 2s + 1\bigr]^{\frac p2}-1-ps}
{\bigl[t^2 + s^2\bigr]^{\frac p2}}\in(0,1].
\end{align*}
\end{cor}

\begin{rem}
By setting $\gamma = 0$ in equation~(\ref{bobkov}), we recover the improved Poincar\'e inequality of Bobkov and Kolonitskii \cite[Theorem 1.5, Equation (1.20)]{bobkov2023improved} with an explicit constant. While the authors observe that the constant can, in principle, be made explicit due to the hidden convexity inequality \cite[Equation (2.10)]{brasco2022comparison}, its precise value is not stated. Here, we make this constant explicit. 
\end{rem}

\begin{cor}\label{cor chinese}
Let $1<p<2\leq n$ and $\lambda, \phi, D$ be from Corollary \ref{main cor}.
\begin{enumerate}
    \item  Then, for all complex-valued $u\in W^{1,p}_{\gamma}(D)$, we have
    \begin{multline*}%\label{thm chi geq c1}
    \int_{D}^{}|\nabla_{\gamma}u|^{p}dz-\lambda\int_{D}^{}|u|^{p}dz\geq c_{1}(p)\int_{D}\left(\left|\nabla_\gamma u\right|+\left|\frac{\nabla_\gamma \phi}{\phi} u\right|\right)^{p-2}\biggl|\nabla_\gamma u - \frac{\nabla_\gamma \phi}{\phi} u\biggr|^{2}dz,
    \end{multline*}
    where $c_1(p)$ is an explicit constant defined by
    \begin{align*}%\label{thm chi c1}
c_1(p) := \inf_{s^2 + t^2 > 0} \frac{\left( t^2 + s^2 + 2s + 1 \right)^{\frac{p}{2}} - 1 - ps}{\left( \sqrt{t^2 + s^2 + 2s + 1} + 1 \right)^{p-2} (t^2 + s^2)} \in \left( 0, \frac{p(p-1)}{2p-1} \right].
\end{align*}
\item Moreover, for all complex-valued $u\in W^{1,p}_{\gamma}(D)$, the remainder term is optimal since
\begin{align*}%\label{thm chi leq c2}
\int_{D}^{}|\nabla_{\gamma}u|^{p}dz-\lambda\int_{D}^{}|u|^{p}dz\leq c_{2}(p)\int_{D}\left(\left|\nabla_\gamma u\right|+\left|\frac{\nabla_\gamma \phi}{\phi} u\right|\right)^{p-2}\biggl|\nabla_\gamma u - \frac{\nabla_\gamma \phi}{\phi} u\biggr|^{2}dz,
\end{align*}
where $c_{2}(p)$ is an explicit constant defined by
\begin{align*}%\label{thm chi c2}
c_2(p) := \sup_{s^2 + t^2 > 0} \frac{\left( t^2 + s^2 + 2s + 1 \right)^{\frac{p}{2}} - 1 - ps}{\left( \sqrt{t^2 + s^2 + 2s + 1} + 1 \right)^{p-2} (t^2 + s^2)} \in \biggl[ \frac{p}{2^{p-1}}, +\infty \biggr).
\end{align*}
\item In addition, for all complex-valued $u\in W^{1,p}_{\gamma}(D)$, we have
\begin{multline*}%\label{thm chi geq c3}
\int_{D}^{}|\nabla_{\gamma}u|^{p}dz-\lambda\int_{D}^{}|u|^{p}dz\geq c_{3}(p)\int_{D}\min\biggl\{\left|\nabla_\gamma u - \frac{\nabla_\gamma \phi}{\phi} u\right|^{p},\\ \left|\frac{\nabla_\gamma \phi}{\phi} u\right|^{p-2}\left|\nabla_\gamma u - \frac{\nabla_\gamma \phi}{\phi} u\right|^{2}\biggr\}dz,
\end{multline*}
where $c_{3}(p)$ is an explicit constant defined by
\begin{multline*}%\label{thm chi c3}
c_3(p) := \min \biggl\{
\inf_{s^2 + t^2 \geq 1} \frac{(t^2 + s^2 + 2s + 1)^{\frac{p}{2}} - 1 - ps}{(t^2 + s^2)^{\frac{p}{2}}},\\
\inf_{0 < s^2 + t^2 < 1} \frac{(t^2 + s^2 + 2s + 1)^{\frac{p}{2}} - 1 - ps}{t^2 + s^2}
\biggr\}\in \left( 0, \frac{p(p-1)}{2} \right].
\end{multline*}
\end{enumerate}
\end{cor}

Before proving Theorem \ref{main thm}, we first need to prove the following complex-valued version of Picone's identity for Baouendi-Grushin vector fields:

\begin{thm}\label{thm1}
Let $u$ be a complex-valued function on $\Omega \subset \mathbb{R}^{m+k}$ and $\phi$ be a non-zero complex-valued function on $\Omega \subset \mathbb{R}^{m+k}$. Then, we have
\begin{align*}
&C_p(\xi, \eta) = |\nabla_\gamma u|^p + (p-1) \left| \frac{\nabla_\gamma \phi}{\phi} u \right|^p - p \, \textnormal{Re} \left[ \left| \frac{\nabla_\gamma \phi}{\phi} u \right|^{p-2} \frac{u}{\phi}(\nabla_\gamma \phi \cdot \overline{\nabla_\gamma u}) \right],
\\&R_p(\xi, \eta) = |\nabla_\gamma u|^p - |\nabla_\gamma \phi|^{p-2} \nabla_\gamma \left( \frac{|u|^p}{|\phi|^{p-2} \phi} \right) \cdot \nabla_\gamma \phi
\end{align*}
and
\begin{align*}
C_p(\xi,\eta)=R_p(\xi,\eta)\geq0,
\end{align*}
where $C_p(\cdot,\cdot)$ is given in (\ref{cp functional intro}) and
\begin{align}\label{notation}
\xi := \nabla_\gamma u, \quad \eta := \nabla_\gamma u - \frac{\nabla_\gamma \phi}{\phi} u.
\end{align}
\end{thm}
\begin{proof}[Proof of Theorem \ref{thm1}] Using the notation (\ref{notation}) in formula (\ref{cp functional intro}), we get
\allowdisplaybreaks
\begin{align*}
C_{p}(\xi,\eta)&=|\nabla_{\gamma}u|^{p}-\left|\frac{\nabla_{\gamma}\phi}{\phi}u\right|^{p}-p\left|\frac{\nabla_{\gamma}\phi}{\phi}u\right|^{p-2}\text{Re}\left(\frac{\nabla_{\gamma}\phi}{\phi}u\right)\cdot \overline{\nabla_{\gamma}u-\frac{\nabla_{\gamma}\phi}{\phi}u}
\\&=|\nabla_{\gamma}u|^{p}-\left|\frac{\nabla_{\gamma}\phi}{\phi}u\right|^{p}-p\text{Re}\left| \frac{\nabla_{\gamma}\phi}{\phi}u \right| ^{p-2}\frac{\nabla_{\gamma}\phi}{\phi}u\cdot \overline{\nabla_{\gamma}u}+p\left| \frac{\nabla_{\gamma}\phi}{\phi}u \right|^{p}
\\&=|\nabla_{\gamma}u|^{p}+(p-1)\left|\frac{\nabla_{\gamma}\phi}{\phi}u\right|^{p}-p\text{Re}\left| \frac{\nabla_{\gamma}\phi}{\phi}u \right| ^{p-2}\frac{u}{\phi}\left( \nabla_{\gamma}\phi \cdot \overline{\nabla_{\gamma}u} \right).
\end{align*}
Now let us define  
\begin{align}\label{R_p}
R_{p}(\xi,\eta):=|\nabla_{\gamma}u|^{p}-\left| \nabla_{\gamma}\phi \right| ^{p-2}\nabla_{\gamma}\left( \frac{|u|^{p}}{|\phi|^{p-2}\phi} \right) \nabla_{\gamma}\phi.
\end{align}
Expanding (\ref{R_p}), we have
\allowdisplaybreaks
\begin{align*}
R_{p}(\xi,\eta)&=|\nabla_{\gamma}u|^{p}-|\nabla_{\gamma}\phi|^{p-2}\Biggl[\frac{p\text{Re}|u|^{p-2}u\overline{\nabla_{\gamma}u}}{|\phi|^{p-2}\phi}+|u|^{p}\nabla_{\gamma}(|\phi|^{2-p}\phi^{-1})\Biggr]\nabla_{\gamma}\phi
\\&=|\nabla_{\gamma}u|^{p}-p\text{Re}\left| \frac{\nabla_{\gamma}\phi}{\phi}u \right| ^{p-2}\frac{u}{\phi}\left( \nabla_{\gamma}\phi \cdot \overline{\nabla_{\gamma}u} \right)
\\&-|\nabla_{\gamma}\phi|^{p-2}\left[|u|^{p}\nabla_{\gamma}(|\phi|^{2-p}\phi^{-1})\right]\nabla_{\gamma}\phi
\\&=|\nabla_{\gamma}u|^{p}-p\text{Re}\left| \frac{\nabla_{\gamma}\phi}{\phi}u \right| ^{p-2}\frac{u}{\phi}\left( \nabla_{\gamma}\phi \cdot \overline{\nabla_{\gamma}u} \right)
\\&-|\nabla_{\gamma}\phi|^{p-2}|u|^{p}\left[(2-p)|\phi|^{-p}\phi(\overline{\nabla_{\gamma}\phi})\phi^{-1}+|\phi|^{2-p}(-1)\phi^{-2}\nabla_{\gamma}\phi\right]\nabla_{\gamma}\phi
\\&=|\nabla_{\gamma}u|^{p}-p\text{Re}\left| \frac{\nabla_{\gamma}\phi}{\phi}u \right| ^{p-2}\frac{u}{\phi}\left( \nabla_{\gamma}\phi \cdot \overline{\nabla_{\gamma}u} \right)+(p-2)\left|\frac{\nabla_{\gamma}\phi}{\phi}u\right|^{p}+\left|\frac{\nabla_{\gamma}\phi}{\phi}u\right|^{p}
\\&=|\nabla_{\gamma}u|^{p}+(p-1)\left|\frac{\nabla_{\gamma}\phi}{\phi}u\right|^{p}-p\text{Re}\left| \frac{\nabla_{\gamma}\phi}{\phi}u \right| ^{p-2}\frac{u}{\phi}\left( \nabla_{\gamma}\phi \cdot \overline{\nabla_{\gamma}u} \right)=C_{p}(\xi,\eta).
\end{align*}
Since $C_p(\xi,\eta)\geq0$ and $C_p(\xi,\eta)=R_p(\xi,\eta)$, this implies that $R_p(\xi,\eta)\geq 0$. 
\end{proof}

Now we are ready to prove Theorem \ref{main thm}.

\begin{proof}[Proof of Theorem \ref{main thm}] Integrating $R_p(\xi,\eta)$ over $\Omega$, we get
\begin{align}\label{intRp0}
\int_{\Omega}R_p(\xi, \eta)dz = \int_{\Omega}|\nabla_\gamma u|^pdz - \int_{\Omega}|\nabla_\gamma \phi|^{p-2} \nabla_\gamma \left( \frac{|u|^p}{|\phi|^{p-2} \phi} \right) \cdot \nabla_\gamma \phi dz \geq 0.
\end{align}
Now let us denote
\begin{align}\label{X}
X := \frac{|u|^{p}}{|\phi|^{p-2}\phi}|\nabla_{\gamma}\phi|^{p-2}\nabla_{\gamma}\phi.
\end{align}
Taking the divergence on both sides in (\ref{X}), we have
\begin{align}\label{divX}
\text{div}_{\nabla_{\gamma}}X = \nabla_{\gamma}\left(\frac{|u|^{p}}{|\phi|^{p-2}\phi}\right)|\nabla_{\gamma}\phi|^{p-2}\nabla_{\gamma}\phi+\frac{|u|^{p}}{|\phi|^{p-2}\phi}\Delta_{\gamma,p}\phi.
\end{align}
Rewriting (\ref{divX}):
\begin{align}\label{key part}
\nabla_{\gamma}\left(\frac{|u|^{p}}{|\phi|^{p-2}\phi}\right)|\nabla_{\gamma}\phi|^{p-2}\nabla_{\gamma}\phi = \nabla_{\gamma}\cdot\left(\frac{|u|^{p}}{|\phi|^{p-2}\phi}|\nabla_{\gamma}\phi|^{p-2}\nabla_{\gamma}\phi\right) - \frac{|u|^{p}}{|\phi|^{p-2}\phi}\Delta_{\gamma,p}\phi.
\end{align}
Substituting (\ref{key part}) in (\ref{intRp0}), we obtain
\begin{multline*}
\int_{\Omega}R_p(\xi, \eta)dz = \int_{\Omega}|\nabla_\gamma u|^pdz - \int_{\Omega}\nabla_{\gamma}\cdot\left(\frac{|u|^{p}}{|\phi|^{p-2}\phi}|\nabla_{\gamma}\phi|^{p-2}\nabla_{\gamma}\phi\right)dz
\\+ \int_{\Omega}\frac{|u|^{p}}{|\phi|^{p-2}\phi}\Delta_{\gamma,p}\phi dz.
\end{multline*}
Using the divergence formula with the fact that $u$ vanishes on the boundary of $\Omega$, we get
\begin{align*}
\int_{\Omega}R_p(\xi, \eta)dz = \int_{\Omega}|\nabla_\gamma u|^pdz+\int_{\Omega}\frac{|u|^{p}}{|\phi|^{p-2}\phi}\Delta_{\gamma,p}\phi dz.
\end{align*}
Since $R_p(\xi,\eta)=C_p(\xi,\eta)$, we have
\begin{align*}
\int_{\Omega}C_p(\xi, \eta)dz = \int_{\Omega}|\nabla_\gamma u|^pdz+\int_{\Omega}\frac{|u|^{p}}{|\phi|^{p-2}\phi}\Delta_{\gamma,p}\phi dz.
\end{align*}
Recalling from \cite[Step 3 of Proof of
Lemma 3.4]{cazacu2024hardy} along with recent results from \cite[Lemma 2.2, 2.3]{CT24}, for $1<p<\infty$, we have
\begin{align*}
C_{p}(\xi,\eta)=0 \iff \eta = \nabla_{\gamma}u - \frac{\nabla_{\gamma}\phi}{\phi}u = 0.
\end{align*}
Taking into account that $\phi\neq0$, we get
\begin{align*}
0=\nabla_{\gamma}u-\frac{\nabla_{\gamma}\phi}{\phi}u=\nabla_{\gamma}\left(\frac{u}{\phi}\right)\phi \iff \nabla_{\gamma}\left(\frac{u}{\phi}\right)\phi = 0 \iff \frac{u}{\phi}=\text{const}.
\end{align*}
Thus, for $1<p<\infty$, the $C_p$-functional vanishes if and only if $\frac{u}{\phi}=\text{const}$. 
\end{proof}

Next, we will proceed with the proof of Corollary \ref{main cor}.

\begin{proof}[Proof of Corollary \ref{main cor}]
Since we assume that the operator $-\Delta_{\gamma,p}$ has a positive eigenvalue $\lambda$ and a corresponding positive eigenfunction $\phi$ on $D$, we have that the pair $(\lambda, \phi)$ satisfies (\ref{eigen}), i.e.,
\begin{align}\label{spec2}
\begin{cases}
- \Delta_{\gamma,p} \phi = \lambda |\phi|^{p-2} \phi & \text{in } D, \\
\phi = 0 & \text{on } \partial D.
\end{cases}
\end{align}
Substituting (\ref{spec2}) to (\ref{main identity}), we have
\begin{align}\label{cp-part}
\int_{D}C_p\left(\nabla_\gamma u,\nabla_\gamma u - \frac{\nabla_\gamma \phi}{\phi} u\right)dz=\int_{D}|\nabla_{\gamma}u|^{p}dz-\lambda\int_{D}|u|^{p}dz,
\end{align}
giving us Part (1). Now for Part (2), we first let $\lambda=\lambda_1>0$ to be the first eigenvalue of $-\Delta_{\gamma,p}$ and $\phi=\phi_1$ be an associated (presumed positive) eigenfunction on $D$. Then, from (\ref{cp-part}), we obtain
\begin{align}\label{cp-part-1}
\int_{D}C_p\left(\nabla_\gamma u,\nabla_\gamma u - \frac{\nabla_\gamma \phi_1}{\phi_1} u\right)dz=\int_{D}|\nabla_{\gamma}u|^{p}dz-\lambda_1\int_{D}|u|^{p}dz
\end{align}
Dropping the remainder term in (\ref{cp-part-1}) and dividing both sides by $\lambda_1>0$, we get
\begin{align}\label{lam1poin}
\int_{D}|u|^{p}dz\leq\frac{1}{\lambda_1}\int_{D}|\nabla_{\gamma}u|^{p}dz,
\end{align}
where
\begin{align}\label{lam1formula}
\lambda_1=
\min_{u\neq0}
\frac{\int_{D}|\nabla_{\gamma}u|^{p}\,dz}{\int_{D}|u|^{p}\,dz}
\end{align}
by \cite[Proposition 4.2]{malanchini2025bifurcation}. To prove that $\frac{1}{\lambda_1}$ is optimal, we define the optimal constant in the $L^{p}$-Poincar\'e inequality as follows:
\begin{align*}
C'=\inf \left\{ M > 0 : \|u\|_{L^p(D)} \leq M \|\nabla_\gamma u\|_{L^p(D)}\right\}
= \sup_{u \neq 0} \frac{\|u\|_{L^p}}{\|\nabla_\gamma u\|_{L^p}}.
\end{align*}
We need to show that $C'=\frac{1}{\lambda_1^{1/p}}$. Taking the infimum over all admissible constants, in (\ref{lam1poin}), we get
\begin{align*}
C'\leq\frac{1}{\lambda_1^{1/p}}.
\end{align*}
Now we note that since $(\lambda_{1}, \phi_{1})$ satisfies (\ref{spec2}), we have
\begin{multline*}
\lambda_{1}\int_{\Omega}|\phi_1|^{p}dz=-\int_{\Omega}\phi_{1}\Delta_{\gamma,p}\phi_{1}dz=-\int_{\Omega}\phi_{1}\text{div}_{\nabla_{\gamma}}\left(|\nabla_{\gamma}\phi_{1}|^{p-2}\nabla_{\gamma}\phi_{1}\right)dz\\=\int_{\Omega}|\nabla_{\gamma}\phi_{1}|^{p}dz,
\end{multline*}
which gives us the following relation:
\begin{align*}
\lambda^{1/p}_{1}=\frac{\|\nabla_{\gamma}\phi_{1}\|_{L^{p}}}{\|\phi_{1}\|_{L^{p}}} \iff \frac{1}{\lambda^{1/p}_1}=\frac{\|\phi_{1}\|_{L^{p}}}{\|\nabla_{\gamma}\phi_{1}\|_{L^{p}}}.
\end{align*}
On the other hand, we have
\begin{align*}
C'=\sup_{u \neq 0} \frac{\|u\|_{L^p}}{\|\nabla_\gamma u\|_{L^p}}\geq\frac{\|\phi_{1}\|_{L^{p}}}{\|\nabla_{\gamma}\phi_{1}\|_{L^{p}}}=\frac{1}{\lambda^{1/p}_1}.
\end{align*}
Therefore, $C'=\frac{1}{\lambda^{1/p}_1}$. The attainability of the constant follows directly from the fact that $C_p$-functional, in (\ref{cp-part-1}), vanishes if and only if $\frac{u}{\phi_1}=\text{const}$.
\end{proof}

\section{Applications}\label{applications}

In this section, we investigate the occurrence of finite-time blow-up and the conditions for global existence of positive solutions to the initial-boundary problem associated with the doubly nonlinear PME involving the $p$-Grushin operator $\Delta_{\gamma, p}$. As a result, we extend \cite[Theorem 1.3 and 1.6]{dukenbayeva2024global} from $p=2$ to $1<p<\infty$. We also refer to \cite{sabitbek2024global}
for similar results when $\gamma=0$.

\subsection{Blow-up solutions of the doubly nonlinear PME} First, we start with the blow-up property.

\begin{thm}\label{app thm1}
Suppose that
\begin{align}\label{cond1pde}
\alpha F(u) \leq u^{\ell}f(u) + \beta u^{p\ell} + \alpha \theta, \quad u > 0,
\end{align}
where
\begin{align*}
F(u) = \frac{p\ell}{\ell + 1} \int_0^u s^{\ell - 1} f(s) \, ds, \quad \ell \geq 1,    
\end{align*}
for some
\begin{align*}
\theta > 0, \quad 0 < \beta \leq \lambda_{1}  \frac{\alpha - \ell - 1}{\ell + 1}  \quad \text{and} \quad \alpha > \ell + 1,  
\end{align*}
where $\lambda_1$ is the first eigenvalue of $-\Delta_{\gamma,p}$. Let the initial data \( u_0 \in L^\infty(D) \cap W_{\gamma}^{1,p}(D) \) satisfy
\begin{align}\label{j0}
J_0 := -\frac{1}{\ell + 1} \int_D |\nabla_{\gamma} u_0^{\ell}|^{p}dz + \int_D \left(F(u_0) - \theta\right)dz > 0.
\end{align}
Then, any positive solution \( u \) of the problem (\ref{pde}) blows up in finite time \( T^* \). That is, there exists
\begin{align}\label{M}
0 < T^* \leq \frac{M}{\sigma \int_D u_0^{\ell+1}dz}
\end{align}
such that
\begin{align*}
\lim_{t \to T^*} \int_0^t \int_D u^{\ell + 1}dzd\tau = +\infty, 
\end{align*}
where \( M > 0 \) and \( \sigma = \frac{\sqrt{p\ell \alpha}}{\ell + 1} - 1 > 0 \). In fact, in (\ref{M}),  we can take
\begin{align*}
M = \frac{(1 + \sigma)(1 + 1/\sigma) \left( \int_D u_0^{\ell + 1} dz\right)^2}{\alpha (\ell + 1) J_0}.    
\end{align*}
\end{thm}
\begin{rem}
If $p=2$ and $\ell=1$ in Theorem \ref{app thm1}, then we recover the results of Suragan and the second author \cite[Theorem 1.5]{suragan2023sharp}.
\end{rem}
\begin{proof}[Proof of Theorem \ref{app thm1}] Throughout the proof, we assume that $u$ is a positive solution to (\ref{pde}). Let us denote
\begin{align}\label{et}
E(t)=\int_0^t \int_{D} u^{\ell+1}dzd\tau + M, \quad t \geq 0,
\end{align}
with some $M>0$ to be chosen later. It suffices to show that
\begin{align}\label{condt et}
E''(t)E(t)-(1+\sigma)(E'(t))^{2}\geq0.
\end{align}
holds for large enough $M>0$. First, we calculate $E'(t)$:
\begin{align*}
E'(t)=\int_{D}u^{\ell+1}dz=(\ell+1)\int_{D}\int_{0}^{t}u^{\ell}u_{\tau}d\tau dz+\int_{D}u^{\ell+1}_{0}dz.
\end{align*}
Then,
\begin{align*}
(E'(t))^{2}&=\left((\ell+1)\int_{D}\int_{0}^{t}u^{\ell} u_{\tau} d\tau dz\right)^{2}+\left(\int_{D}u^{\ell+1}_{0} dz\right)^{2}
\\&+2(\ell+1)\left(\int_{D}\int_{0}^{t}u^{\ell} u_{\tau} d\tau dz\right)\left(\int_{D}u^{\ell+1}_{0}dz\right).
\end{align*}
By utilizing H\"older and Cauchy-Schwarz inequalities, we obtain
\allowdisplaybreaks
\begin{align}
(E'(t))^{2}&\leq (\ell+1)^{2}(1+\delta)\left(\int_{D}\int_{0}^{t}u^{\ell} u_{\tau}d\tau dz\right)^{2} \nonumber 
\\&+\left(1+\frac{1}{\delta}\right)\left(\int_{D}u^{\ell+1}_{0}dz\right)^{2} \nonumber
\\&= (\ell + 1)^2(1 + \delta) \left( \int_D \int_0^t u^{(\ell + 1)/2 + (\ell - 1)/2} u_\tau d\tau dz \right)^2 \nonumber 
\\& + \left(1 + \frac{1}{\delta} \right) \left( \int_D u_0^{\ell + 1} dz \right)^2 \nonumber
\\&\leq (\ell + 1)^2(1 + \delta) \left( \int_D \left( \int_0^t u^{\ell + 1} d\tau \right)^{\frac{1}{2}} \left( \int_0^t u^{\ell - 1} u_\tau^2 d\tau \right)^{\frac{1}{2}} dz \right)^2 \nonumber
\\& + \left(1 + \frac{1}{\delta} \right) \left( \int_D u_0^{\ell + 1} dz \right)^2 \nonumber
\\&\leq (\ell + 1)^2(1 + \delta) \left( \int_0^t \int_D u^{\ell + 1} dz d\tau \right) \left( \int_0^t \int_D u^{\ell - 1} u_\tau^2 dz d\tau \right) \nonumber
\\& + \left(1 + \frac{1}{\delta} \right) \left( \int_D u_0^{\ell + 1} dz \right)^2 \label{e't^2}
\end{align}
for any $\delta>0$. By a similar procedure, we obtain $E''(t)$:
\begin{align}\label{e'' part 1}
E''(t) &= (\ell + 1) \int_D u^\ell u_t dz .
\end{align}
Substituting $u_{t}=\Delta_{\gamma,p}(u^{\ell})+f(u)$ from (\ref{pde}) into (\ref{e'' part 1}) and integrating by parts, we get
\begin{align*}
E''(t) &= (\ell+1)\int_{D}u^{\ell} \Delta_{\gamma,p}(u^{\ell})dz+(\ell+1)\int_{D}u^{\ell}f(u)dz
\\&= -(\ell + 1) \int_D |\nabla_\gamma u^\ell|^{p}dz + (\ell + 1) \int_D u^\ell f(u)dz.
\end{align*}
Now we apply the condition (\ref{cond1pde}) and the Poincar\'e inequality (\ref{poincare lp ineq}):
\begin{align}
E''(t)&\geq (\ell + 1) \int_D \left( \alpha F(u) - \beta u^{p\ell}  - \alpha\theta \right) dz - (\ell + 1) \int_D |\nabla_\gamma u^\ell|^{p}dz  \nonumber
\\&= \alpha(\ell + 1) \left( -\frac{1}{\ell + 1} \int_D |\nabla_\gamma u^\ell|^{p}dz + \int_D \left( F(u) - \theta \right) dz \right) \nonumber
\\& + (\alpha - \ell - 1) \int_D |\nabla_\gamma u^\ell|^{p}dz - \beta(\ell + 1) \int_D u^{p\ell} dz \nonumber
\\&\geq \alpha(\ell + 1) \left( -\frac{1}{\ell + 1} \int_D |\nabla_\gamma u^\ell|^{p}dz + \int_D \left( F(u) - \theta \right) dz \right) \nonumber
\\& + \left( \lambda_{1}(\alpha - \ell - 1) - \beta(\ell + 1) \right) \int_D u^{p\ell} dz \nonumber
\\&\ge \alpha(\ell + 1) \left( -\frac{1}{\ell + 1} \int_D |\nabla_\gamma u^\ell|^{p}dz + \int_D \left( F(u) - \theta \right) dz \right) \nonumber
\\&=: \alpha(\ell + 1) J(t), \label{j(t) part 1}
\end{align}
where
\begin{align}
J(t) &= J(0) + \int_0^t \frac{dJ(\tau)}{d\tau}d\tau  \nonumber 
\\&=J(0) - \frac{1}{\ell + 1} \int_0^t \int_D \frac{d}{d\tau}|\nabla_\gamma u^\ell|^pdz d\tau 
+ \int_0^t \int_D \frac{d}{d\tau}(F(u) - \theta)dz d\tau \nonumber
\\&= J(0) - \frac{p}{\ell + 1} \int_0^t \int_D |\nabla_\gamma u^\ell|^{p-2}\nabla_{\gamma}u^{\ell} \cdot \nabla_\gamma(u^\ell)_\tau dz d\tau \nonumber
\\& + \int_0^t \int_D F_u(u)u_\tau dz d\tau \nonumber
\\&= J(0) + \frac{p}{\ell + 1} \int_0^t \int_D (\Delta_{\gamma,p}(u^\ell) + f(u))(u^\ell)_\tau dz d\tau \nonumber
\\&= J(0) + \frac{p\ell}{\ell + 1} \int_0^t \int_D u^{\ell - 1} u_\tau^2 dz d\tau. \label{j(t) part 2}
\end{align}
Combining (\ref{j(t) part 1}) and (\ref{j(t) part 2}), we get 
\begin{align}\label{e'' part 2}
E''(t)\geq \alpha(\ell+1)J(0)+p \ell\alpha  \int_0^t \int_D u^{\ell - 1} u_\tau^2 dz d\tau.
\end{align}
We note that $J_{0}$ from (\ref{j0}) is actually equal to $J(0)$. Since $\alpha>\ell+1$, we have that $\sigma=\delta=\frac{\sqrt{p\ell\alpha}}{\ell+1}-1>0$. Putting (\ref{et}), (\ref{e't^2}) and (\ref{e'' part 2}) to (\ref{condt et}), we obtain
\begin{align*}
&E''(t)E(t) - (1 + \sigma)(E'(t))^2 
\geq \biggl(\alpha(\ell+1)J(0)+p \ell\alpha\int_0^t \int_D u^{\ell - 1} u_\tau^2 dz d\tau\biggr)\\&\times\left(\int_0^t \int_{D} u^{\ell+1} dz d\tau + M\right) - (\ell + 1)^2(1+\sigma)(1 + \delta) \left( \int_0^t \int_D u^{\ell + 1} dz d\tau \right) 
\\&\times\left( \int_0^t \int_D u^{\ell - 1} u_\tau^2 dz d\tau \right) - (1+\sigma)\left(1 + \frac{1}{\delta} \right) \left( \int_D u_0^{\ell + 1} dz \right)^2
\\&\geq \alpha M(\ell + 1)J(0)  + p\ell\alpha \left( \int_0^t \int_D u^{\ell+1} dz d\tau \right) \left( \int_0^t \int_D u_\tau^2 u^{\ell-1} dz d\tau \right) 
\\& - (\ell + 1)^2(1 + \sigma)(1 + \delta) \left( \int_0^t \int_D u^{\ell+1}dz d\tau \right) 
\left( \int_0^t \int_D u^{\ell-1}u_\tau^2 dz d\tau \right) 
\\& - (1 + \sigma)\left(1 + \frac{1}{\delta}\right) \left( \int_D u_0^{\ell+1} dz \right)^2 
\\&\geq \alpha M(\ell + 1)J(0) - (1 + \sigma)\left(1 + \frac{1}{\delta}\right) \left( \int_D u_0^{\ell+1} dz \right)^2.
\end{align*}
Since $J(0)>0$, we can choose $M$ to be large enough such that we have (\ref{condt et}). In particular, we can take
\begin{align*}
M = \frac{(1 + \sigma)(1 + 1/\sigma) \left( \int_D u_0^{\ell + 1}dz\right)^2}{\alpha (\ell + 1) J(0)}.   
\end{align*}
However, it also means that for $t\geq0$
\begin{align*}
\frac{d}{dt} \left( \frac{E'(t)}{E^{\sigma + 1}(t)} \right) \geq 0 
\Rightarrow 
\begin{cases}
E'(t) \geq \left( \frac{E'(0)}{E^{\sigma + 1}(0)} \right) E^{1 + \sigma}(t), \\
E(0) = M.
\end{cases}
\end{align*}
Considering $\sigma = \frac{\sqrt{p\ell\alpha}}{\ell + 1} - 1 > 0$, we obtain
\begin{align*}
- \frac{1}{\sigma} \left( E^{-\sigma}(t) - E^{-\sigma}(0) \right) \geq \frac{E'(0)}{E^{\sigma + 1}(0)} t,    
\end{align*}
which gives together with $E(0) = M$ that
\begin{align*}
E(t) \geq \left( \frac{1}{M^{\sigma}} - \frac{\sigma\int_{D}u^{\ell+1}_{0} dz}{M^{\sigma+1}}t\right)^{-\frac{1}{\sigma}}.
\end{align*}
Thus, we have observed that the blow-up time $T^*$ satisfies
\begin{align*}
0 < T^* \leq \frac{M}{\sigma \int_D u_0^{\ell + 1}dz},
\end{align*}
completing the proof.
\end{proof}

\subsection{Global existence for the doubly nonlinear PME} In this section, we show that under some assumptions, if a positive solution to (\ref{pde}) exists, its norm is globally controlled.
\allowdisplaybreaks
\begin{thm}\label{app thm2}
Assume that
\begin{align}\label{cond2pde}
\alpha F(u) \geq u^{\ell} f(u) + \beta u^{p\ell} + \alpha \theta, \quad u > 0,
\end{align}
where
\begin{align*}
F(u) = \frac{p\ell}{\ell+1} \int_0^u s^{\ell - 1} f(s)\,ds, \quad \ell \geq 1,
\end{align*}
for some
\begin{align*}
\theta \geq 0, \quad \alpha \leq 0 \quad \text{and} \quad \beta \geq \lambda_{1} \frac{\alpha - \ell - 1}{\ell + 1},    
\end{align*}
where $\lambda_1$ is the first eigenvalue of $-\Delta_{\gamma,p}$. Assume also that the initial data \( u_0 \in L^{\infty}(D) \cap W^{1,p}_{\gamma}(D) \) satisfies the inequality
\begin{align*}
J_0 := \int_D \left( F(u_0) - \theta \right)dz 
- \frac{1}{\ell + 1} \int_D |\nabla_{\gamma} u_0^{\ell}|^{p}dz > 0.   
\end{align*}
If \( u \) is a positive local solution of the problem (\ref{pde}), then it is global with the property
\begin{align*}
\int_D u^{\ell + 1} dz \leq \int_D u_0^{\ell + 1} dz. 
\end{align*}
\end{thm}
\begin{proof}[Proof of Theorem \ref{app thm2}] Here, let us define
\begin{align*}
\mathcal{E}(t) = \int_D u^{\ell + 1}dz.    
\end{align*}
Applying (\ref{cond2pde}), Poincar\'e inequality (\ref{poincare lp ineq}) and \(\beta \geq \lambda_{1} \frac{\alpha - \ell - 1}{\ell + 1}\), we have
\begin{align*}
\mathcal{E}'(t) &= (\ell+1)\int_{D}u^{\ell}\Delta_{\gamma,p}(u^{\ell})dz+(\ell+1)\int_{D}u^{\ell}f(u)dz
\\&= (\ell + 1)\left(-\int_D |\nabla_\gamma u^\ell|^{p}dz + \int_D u^{\ell} f(u)dz\right)
\\&\leq(\ell+1)\left(-\int_D |\nabla_\gamma u^\ell|^{p}dz + \int_D \left( \alpha F(u) - \beta u^{p\ell} - \alpha\theta \right) dz\right)
\\&=\alpha(\ell + 1) \left( -\frac{1}{\ell + 1} \int_D |\nabla_\gamma u^\ell|^{p}dz + \int_D \left( F(u) - \theta \right) dz \right)
\\&- (\ell + 1 - \alpha) \int_D |\nabla_\gamma u^\ell|^{p}dz - \beta(\ell + 1) \int_D u^{p\ell}dz
\\&\leq \alpha(\ell + 1) \left( -\frac{1}{\ell + 1} \int_D |\nabla_\gamma u^\ell|^{p}dz + \int_D \left( F(u) - \theta \right) dz \right)
\\&- \left( \lambda_{1}(\ell + 1 - \alpha) + \beta(\ell + 1) \right) \int_D u^{p\ell}dz
\\&\leq \alpha(\ell + 1) \left( -\frac{1}{\ell + 1} \int_D |\nabla_\gamma u^{\ell}|^{p}dz + \int_D \left( F(u) - \theta \right) dz \right)
\\&=\alpha(\ell + 1) J(t),
\end{align*}
where the functional $J(t)$ is taken from the proof of Theorem \ref{app thm1}. Considering (\ref{j(t) part 2}) and the fact that $\alpha\leq0$ in $\mathcal{E}'(t)$, we get
\begin{align*}
\mathcal{E}'(t) \leq \alpha(\ell + 1)J(0) + p\ell\alpha  \int_0^t \int_D u^{\ell - 1}u_\tau^{2}dzd\tau \leq 0,    
\end{align*}
which implies
\begin{align*}
\mathcal{E}(t) \leq \mathcal{E}(0),    
\end{align*}
completing the proof. 
\end{proof}

\bibliographystyle{alpha}
\bibliography{citation}

\end{document}